\documentclass[12pt,a4paper,oneside]{article}
\usepackage[T1]{fontenc}
\usepackage{graphicx}
\usepackage{mathtools}
\usepackage{amssymb}
\usepackage{amsthm}
\usepackage{thmtools}
\usepackage{xcolor}
\usepackage{nameref}
\usepackage{hyperref}
\newtheorem{theorem}{Theorem}[section]
\newtheorem{corollary}[theorem]{Corollary}
\newtheorem{lemma}[theorem]{Lemma}
\newtheorem{question}[theorem]{Question}
\theoremstyle{definition}
\newtheorem{definition}[theorem]{Definition}

\begin{document}
	
\title{  Essential $\mathcal{F}$-sets of $\mathbb{N}$ under nonhomogeneous spectra}

\date{}
\author{Pintu Debnath
	\footnote{Department of Mathematics,
		Basirhat College,
		Basirhat-743412, North 24th parganas, West Bengal, India.\hfill\break
		{\tt pintumath1989@gmail.com}}
}
\maketitle
	
\begin{abstract}
	 Let $\alpha>0$ and $0<\gamma<1$. Define $g_{\alpha,\gamma}:\mathbb{N}\rightarrow\mathbb{N}$
	by $g_{\alpha,\gamma}\left(n\right)=\lfloor\alpha n+\gamma\rfloor$.
	The set $\left\{ g_{\alpha,\gamma}\left(n\right):n\in\mathbb{N}\right\} $
	is called the nonhomogeneous spectrum of $\alpha$ and $\gamma$.
	We refer to the maps $g_{\alpha,\gamma}$ as nonhomogeneous spectra. In \cite{BHK}, 
	Bergelson, Hindman and Kra showed that if $A$ is an $IP$-set, a
	central set, an $IP^{\star}$-set, or a central$^{\star}$-set, then
	$g_{\alpha,\gamma}\left[A\right]$ is the corresponding objects. Hindman
	and Johnsons extended this result to include several other notions
	of largeness: $C$-sets, $J$-sets, strongly central sets, piecwise
	syndetic sets,  $AP$-sets syndetic set, $C^{\star}$-sets, strongly
	central$^{\star}$- sets . In \cite{DHS},  De, Hindman and Strauss introduced $C$-set
	and $J$-set and showed that $C$-sets satisfy the conclusion of the
	Central Sets Theorem. To prepare this article, we have been strongly   motivated by the fact that $C$-sets are essential
	$\mathcal{J}$-sets. In this article,  we prove some new results regarding nonhomogeneous
	spectra of essential $\mathcal{F}$-sets for shift invariant $\mathcal{F}$-sets. We have  a special interest in the family $\mathcal{HSD}=\left\{ A\subseteq\mathbb{N}:\sum_{n\in A}\frac{1}{n}=\infty\right\} $ as this family is directly connected with the  famous Erdős sum of reciprocal conjecture and as a consequence we get $g_{\alpha,\gamma}\left[\mathbb{P}\right]\in\mathcal{HSD}$, where $\mathbb{P}$ is the set of prime numbers in $\mathbb{N}$. Throughout this article, we use some elementary techniques and algebra of the Stone-\v{C}ech compactifications of discrete semigroups.
\end{abstract}
\section{Introduction}
Given a positive real number $\alpha$, the set $\left\{ \lfloor n\alpha\rfloor:n\in\mathbb{N}\right\} $
is called the spectrum of $\alpha$. Numerous results about spectra
were derived by Skolem \cite{Sk} and Bang \cite{B}. In \cite{Sk2}
Skolem introduced the more general sets $\left\{ \lfloor n\alpha+\gamma\rfloor:n\in\mathbb{N}\right\} $,
determining for example when two such sets can be disjoint. In \cite{BHK}
V. Bergelson , N. Hindman and B. Kra showed that if $\alpha>0$, $0<\gamma<1$,
and a subset $A$ is an $IP$-set, a central set, an $IP^{\star}$-set,
or a central$^{\star}$-set, then $g_{\alpha,\gamma}\left[A\right]$
is the corresponding objects. It is trivial for each $n\in\mathbb{N}$,
$n\mathbb{N}$ is an $IP^{\star}$-set. It is not so trivial that
$\left\{ \lfloor\sqrt{7}\lfloor n\pi+\frac{1}{e}\rfloor+\frac{2}{3}\rfloor:n\in\mathbb{N}\right\} $
is an $IP^{\star}$-set. In \cite{HJ} Hindman and Johnson proved
that $C$-sets, $J$-sets, strongly central sets, piecwise syndetic
sets $AP$-sets syndetic set, $C^{\star}$-sets, strongly central$^{\star}$-
sets preserve under nonhomogeneous spectra. Before the end of this section we will mention what we want to do but we need to discuss some definitions and algebraic preliminaries first.

A collection $\mathcal{F\subseteq P}\left(S\right)\setminus\left\{ \emptyset\right\} $
is upward hereditary if whenever $A\in\mathcal{F}$ and $A\subseteq B\subseteq S$
then it follows that $B\in\mathcal{F}$. A nonempty and upward hereditary
collection $\mathcal{F\subseteq P}\left(S\right)\setminus\left\{ \emptyset\right\} $
will be called a family. If $\mathcal{F}$ is a family, the dual family
$\mathcal{F}^{*}$ is given by, $\mathcal{F}^{*}=\{E\subseteq S:\forall A\in\mathcal{F},E\cap A\neq\emptyset\}$.
A family $\mathcal{F}$ possesses the Ramsey property if whenever
$A\in\mathcal{F}$ and $A=A_{1}\cup A_{2}$ there is some $i\in\left\{ 1,2\right\} $
such that $A_{i}\in\mathcal{F}$.

Before giving some examples of families with Ramsey property, 
We give some definitions that have noncommutative versions, but here we define only for commutative semigroups those are enough for this article.
\begin{definition}
	Let $\left(S,+\right)$ be a commutative semigroup and $A\subseteq S$.
	\begin{enumerate}
		\item The set $A$ is $IP$-set if and only if  there exists a sequence $\left\{ x_{n}\right\} _{n=1}^{\infty}$ in $S$ such that $FS\left(\left\{ x_{n}\right\} _{n=1}^{\infty}\right)\subseteq A$. Where $$ FS\left(\left\{ x_{n}\right\} _{n=1}^{\infty}\right)=\left\{ \sum_{n\in F}x_{n}:F\in\mathcal{P}_{f}\left(\mathbb{N}\right)\right\}\,. $$
		\item The set $A$ is piecewise syndetic set if there is a some $G\in\mathcal{P}_{f}\left(S\right)$ such that for any $F\in\mathcal{P}_{f}\left(S\right)$ there is some $x\in S$ such that $$F+x\subset \bigcup_{t\in G}\left(-t+A\right)\,,$$ where $-t+A=\left\{ s\in S:t+s\in A\right\}$.
		\item The set $A$ is $J$-set if  for every $F\in\mathcal{P}_{f}\left({}^\mathbb{N}S\right)$, there exist $a\in S$ and $H\in\mathcal{P}_{f}\left(\mathbb{N}\right)$ such that for each $f\in F$, $$a+\sum_{n\in H}f(n)\in A\,.$$
	\end{enumerate}
\end{definition}

Let us talk about positive density sets of $\mathbb{N}$. Unmodified
positive density standardly means positive asymptotic density, and
set of subsets of $\mathbb{N}$ with positive asymptotic density is
not a family since it is not closed under passage to supersets. (The
asymptotic density of $A\subseteq\mathbb{N}$ is $$d\left(A\right)=lim_{n\rightarrow\infty}\frac{\mid A\cap\left\{ 1,2,\ldots,n\right\} \mid}{n}$$
provided that limit exists and undefined otherwise.). But $\mathcal{F}$
is the family of subsets of $\mathbb{N}$ with positive upper asymptotic
density.
\begin{definition}
	Let $A$ be a subset of $\mathbb{N}$
	\begin{enumerate}
		\item The set $A$ is positive upper asymptotic  density if $\overline{d}\left(A\right)>0$, where $$\overline{d}\left(A\right)=\limsup_{n\rightarrow\infty}\frac{\mid A\cap\left\{ 1,2,\ldots,n\right\} \mid}{n}\,.$$
		\item The set $A$ is positive upper banach  density if  $d^\star\left(A\right)>0$, where $$d^\star\left(A\right)=\limsup_{(n-m)\rightarrow\infty}\frac{\mid A\cap\left\{ n+1,n+2,\ldots,m\right\} \mid}{n-m}\,.$$
		
	\end{enumerate}
\end{definition}

There are many families $\mathcal{F}$ with Ramsay property, where
for uniformity we consider following families for $\mathbb{\mathbb{N}}$.
\begin{itemize}
	\item The infinite sets or $IF$-sets,
	\item The piecewise syndetic sets or $PS$-sets,
	\item The sets of positive upper asymptotic density or $PUAD$-sets,
	\item The sets of positive upper banach density or $PUBD$-sets,
	\item The set containing arbitrary large arithmetic progression or $AP$-sets
	\item The set with property that $\sum_{n\in A}\frac{1}{n}=\infty$ or $HSD$-sets 
	\item The $J$-sets,
	\item The $IP$-sets.
\end{itemize}
We give a brief review of algebraic structure of the Stone-\v{C}ech
compactification of a semigroup $\left(S,\cdot\right)$, not necessarily commutative with the discrete topology.

The set $\{\overline{A}:A\subset S\}$ is a basis for the closed sets
of $\beta S$. The operation `$\cdot$' on $S$ can be extended to
the Stone-\v{C}ech compactification $\beta S$ of $S$ so that $(\beta S,\cdot)$
is a compact right topological semigroup (meaning that for each    $p\in\beta$ S the function $\rho_{p}\left(q\right):\beta S\rightarrow\beta S$ defined by $\rho_{p}\left(q\right)=q\cdot p$ 
is continuous) with $S$ contained in its topological center (meaning
that for any $x\in S$, the function $\lambda_{x}:\beta S\rightarrow\beta S$
defined by $\lambda_{x}(q)=x\cdot q$ is continuous). This is a famous
Theorem due to Ellis that if $S$ is a compact right topological semigroup
then the set of idempotents $E\left(S\right)\neq\emptyset$. A nonempty
subset $I$ of a semigroup $T$ is called a $\textit{left ideal}$
of $S$ if $TI\subset I$, a $\textit{right ideal}$ if $IT\subset I$,
and a $\textit{two sided ideal}$ (or simply an $\textit{ideal}$)
if it is both a left and right ideal. A $\textit{minimal left ideal}$
is the left ideal that does not contain any proper left ideal. Similarly,
we can define $\textit{minimal right ideal}$ and $\textit{smallest ideal}$.

Any compact Hausdorff right topological semigroup $T$ has the smallest
two sided ideal

$$
\begin{aligned}
	K(T) & =  \bigcup\{L:L\text{ is a minimal left ideal of }T\}\\
	 &=  \bigcup\{R:R\text{ is a minimal right ideal of }T\}.
\end{aligned}$$

Given a minimal left ideal $L$ and a minimal right ideal $R$, $L\cap R$
is a group, and in particular contains an idempotent. If $p$ and
$q$ are idempotents in $T$ we write $p\leq q$ if and only if $pq=qp=p$.
An idempotent is minimal with respect to this relation if and only
if it is a member of the smallest ideal $K(T)$ of $T$. Given $p,q\in\beta S$
and $A\subseteq S$, $A\in p\cdot q$ if and only if the set $\{x\in S:x^{-1}A\in q\}\in p$,
where $x^{-1}A=\{y\in S:x\cdot y\in A\}$. See \cite{HS} for
an elementary introduction to the algebra of $\beta S$ and for any
unfamiliar details.

It will be easy to check that the family$\mathcal{F}$ has the Ramsey
property iff the family $\mathcal{F}^{*}$ is a filter. For a family
$\mathcal{F}$ with the Ramsey property, let $\beta(\mathcal{F})=\{p\in\beta S:p\subseteq\mathcal{F}\}$.
Then the following from \cite[Theorem 5.1.1]{C}:
\begin{theorem}
	Let $S$ be a discrete set. For every family $\mathcal{F\subseteq P}\left(S\right)$
	with the Ramsay property, $\beta\left(\mathcal{F}\right)\subseteq\beta S$
	is closed. Furthermore, $\mathcal{F}=\cup\beta\left(\mathcal{F}\right)$.
	Also if $K\subseteq\beta S$ is closed, $\mathcal{F}_{K}=\left\{ E\subseteq S:\overline{E}\cap K\neq\emptyset\right\} $
	is a family with the Ramsay property and $\overline{K}=\beta\left(\mathcal{F}_{K}\right)$.
\end{theorem}

Let $S$ be a discrete semigroup, then for every family $\mathcal{F\subseteq P}\left(S\right)$
with the Ramsay property, $\beta\left(\mathcal{F}\right)\subseteq\beta S$
is closed. If $\beta\left(\mathcal{F}\right)$ be a subsemigroup of
$\beta S$, then $E\left(\beta\mathcal{F}\right)\neq\emptyset$. But
may not be subsemigroup. For example, let $\mathcal{F}=\mathcal{IP}$,
the family of IP- sets. It is easy to show that $\beta\left(\mathcal{F}\right)=\beta\left(\mathcal{IP}\right)=E\left(\beta S\right)$.
But $E\left(\beta S\right)$ is not a subsemigroup of $\beta S$.
\begin{definition}[Essential $\mathcal{F}$-set]
	Let $\mathcal{F}$ be a family with Ramsay property such that $\beta(\mathcal{F})$
	is a subsemigroup of $\beta S$ and $p$ be an idempotent in $\beta(\mathcal{F})$,
	then each member of $p$ is called essential $\mathcal{F}$-set. And
	$A\subset S$ is called essential $\mathcal{F}^{\star}$-set if $A$
	intersects with all essential $\mathcal{F}$-sets.
\end{definition}

The family $\mathcal{F}$ is called left  shift-invariant if
for all $s\in S$ and all $E\in\mathcal{F}$ one has $sE\in\mathcal{F}$ ($sE=\left\{ st:t\in E\right\} $).
The family $\mathcal{F}$ is called left inverse shift-invariant
if for all $s\in S$ and all $E\in\mathcal{F}$ one has $s^{-1}E\in\mathcal{F}$ ($s^{-1}E=\left\{ t:st\in E\right\} $).
From \cite[Theorem 5.1.2]{C}:
\begin{theorem}
	If $\mathcal{F}$ is a family having the Ramsey property then $\beta\mathcal{F}\subseteq\beta S$
	is a left ideal if and only if $\mathcal{F}$ is left shift-invariant.
	
\end{theorem}

From\cite[Theorem 5.1.10]{C}, we can identifie those families
$\mathcal{F}$ with Ramsey property for which $\beta\left(\mathcal{F}\right)$
is a subsemigroup of $\beta S$ . The condition is a rather technical
weakening of left shift-invariance.
\begin{theorem}
	Let $S$ be any semigroup, and let $\mathcal{F}$ be a family of subsets
	of $S$ having the Ramsey property. Then the following are equivalent:
	
	\begin{itemize}
	
	\item[(1)]  $\beta\left(\mathcal{F}\right)$ is a subsemigroup of $\beta S$.
	
	\item[(2)]  $\mathcal{F}$ has the following property: If $E\subseteq S$
	is any set, and if there is $A\in\mathcal{F}$ such that for all finite
	$H\subseteq A$ one has $\left(\cap_{q\in H}x^{-1}E\right)\in\mathcal{F}$,
	then $E\in\mathcal{F}$.
	
    \end{itemize}
\end{theorem}

Let us abbreviate
\begin{itemize}
	\item the family of infinite sets as $\mathcal{IF}$,
	\item the family of piecewise syndetic sets as $\mathcal{PS}$,
	\item the family of  positive upper asymptotic density as $\mathcal{PUAD}$
	\item the family of  positive upper asymptotic density as $\mathcal{PUBD}$
	\item the family of sets containing arithmetic progression of arbitrary length as $\mathcal{AP}$,
	\item the family of sets with the property $\sum_{n\in A}\frac{1}{n}=\infty$ as $\mathcal{HSD}$,
	\item and the family of $J$-sets as $\mathcal{J}$.
\end{itemize}

From the definition of essential $\mathcal{F}$-set together with the abbreviations, we get,
\begin{itemize}
	\item $IP$ set is an essential $\mathcal{IF}$-set,
	\item quasi central set is an essential $\mathcal{PS}$-set,
	\item $D$-set is an essential $\mathcal{PUBD}$-set,
	\item and $C$-set is an essential $\mathcal{J}$-set.
\end{itemize}

We state the following theorem without proof follows from \cite[Theorem 3.11]{HS}.
\begin{theorem}
	Let $S$ be a discrete set. For every family $\mathcal{F\subseteq P}\left(S\right)$
	with the Ramsay property, $\beta\left(\mathcal{F}\right)\subseteq\beta S$,
	$A$ is a $\mathcal{F}$-set if and only if $\overline{A}\cap\beta\left(\mathcal{F}\right)\neq\emptyset$.
\end{theorem}

From the above theorem, we get the following corollary:
\begin{corollary}
	Let $A\subseteq\mathbb{N}$.
	\begin{itemize}
		\item[(a)] $A$ is a $AP$-set if and only if $\overline{A}\cap\beta\left(\mathcal{\mathcal{AP}}\right)\neq\emptyset$.
		\item[(b)] $A$ is $IF$-set if and only if $\overline{A}\cap\beta\left(\mathcal{\mathcal{IF}}\right)\neq\emptyset$.
		\item[(c)] $A$ is a $PS$-set if and only if $\overline{A}\cap\beta\left(\mathcal{\mathcal{PS}}\right)\neq\emptyset$.
		\item[(d)] $A$ is a $HSD$-set if and only if $\overline{A}\cap\beta\left(\mathcal{\mathcal{HSD}}\right)\neq\emptyset$.
		\item[(e)] $A$ is $J$-set if and only if $\overline{A}\cap\beta\left(\mathcal{\mathcal{J}}\right)\neq\emptyset$.
		\item[(f)]$A$ is set with $PUAD$-set if and only if $\overline{A}\cap\beta\left(\mathcal{\mathcal{PUAD}}\right)\neq\emptyset$.
		\item[(f)]$A$ is set with $PUBD$-set if and only if $\overline{A}\cap\beta\left(\mathcal{\mathcal{PUBD}}\right)\neq\emptyset$.
		\item[(g)] $A$ is a $AP^{\star}$-set if and only if $\beta\left(\mathcal{\mathcal{AP}}\right)\subseteq\overline{A}\neq\emptyset$.
		\item[(h)]$A$ is $IF^{\star}$-set if and only if $\beta\left(\mathcal{\mathcal{IF}}\right)\subseteq\overline{A}\neq\emptyset$.
		\item[(i)] $A$ is a $PS^{\star}$-set if and only if $\beta\left(\mathcal{\mathcal{PS}}\right)\subseteq\overline{A}\neq\emptyset$.
		\item[(j)] $A$ is a $HSD^{\star}$-set if and only if $\beta\left(\mathcal{\mathcal{HSD}}\right)\subseteq\overline{A}\neq\emptyset$.
		\item[(k)] $A$ is $J^{\star}$-set if and only if $\beta\left(\mathcal{\mathcal{J}}\right)\subseteq\overline{A}\neq\emptyset$.
		\item[(l)] $A$ is set with $PUAD^{\star}$-set if and only if $\beta\left(\mathcal{\mathcal{PUAD}}\right)\subseteq\overline{A}\neq\emptyset$.
		\item[(m)] $A$ is set with $PUBD^{\star}$-set if and only if $\beta\left(\mathcal{\mathcal{PUBD}}\right)\subseteq\overline{A}\neq\emptyset$.
	\end{itemize}
	
\end{corollary}
In Section \ref{Elementary results}, we will present elementary proofs regarding preservation of $\mathcal{HSD}$ family and $\mathcal{PUAD}$ family under nonhomogeneous spectra.
In Section \ref{Algebraic results}, we will proof preservation of nonhomogeneous spectra of special types of essential $\mathcal{F}$-sets. 

\section{Some elementary results}\label{Elementary results}

 We get the following result from \cite[Theorem 2.4, Theorem 2.5]{HJ}.
 \begin{theorem}\label{spectra of PS and AP}
 	 Let $\alpha>0$, and let $0<\gamma<1$ and $A$ be a subset of $\mathbb{N}$.
 \begin{itemize}
 	\item[(1)] If $A$ is piecwise syndetic subset of $\mathbb{N}$, then $g_{\alpha,\gamma}\left[A\right]$ is picewise syndetic i.e., $g_{\alpha,\gamma}\left[\mathcal{PS}\right]\subseteq\mathcal{PS}$.
 	
 	\item[(2)]  If A is $AP$-set of $\mathbb{N}$, then $g_{\alpha,\gamma}\left[A\right]$ is $AP$-set i.e., $g_{\alpha,\gamma}\left[\mathcal{AP}\right]\subseteq\mathcal{AP}$.
 \end{itemize}	
 	
 \end{theorem}
 
 The Erdős–Turán conjecture states that if the sum of the reciprocals (harmonic sum)   of the members of a set $A$ of $\mathbb{N}$ diverges, then $A$ contains arbitrarily long arithmetic progressions. This conjecture establishes the importance of $\mathcal{HSD}$ family in additive combinatorics. The following theorem states the preservation of $\mathcal{HSD}$ family under nonhomogeneous spectra.

\begin{theorem}\label{Spectra of HSD}
	Let $\alpha>0$, and let $0<\gamma<1$. Then $g_{\alpha,\gamma}\left[\mathcal{\mathcal{HSD}}\right]\subseteq\mathcal{\mathcal{\mathcal{HSD}}}$.
\end{theorem}

\begin{proof}
	Let $A$ be a subset of $\mathbb{N}$ such that $A\in\mathcal{\mathcal{HSD}}$.
	For $\alpha>0$ , find out $r\in\mathbb{N}$ , such that $r\alpha>1$.
	Let $A=\left\{ n_{r}\right\} _{r=1}^{\infty}$ and partition $A=\cup_{i=1}^{r}A_{i}$
	where $A_{i}=\left\{ n_{i+rj}\right\} _{j=0}^{\infty}$. Then there
	exists $i\in\left\{ 1,2,\ldots,r\right\} $ such that $$f:A_{i}\rightarrow g_{\alpha,\gamma}\left[A_{i}\right]$$
	defined by $f\left(n\right)=\left[n\alpha+\gamma\right]$ is injective
	and increasing function and $A_{i}\in\mathcal{\mathcal{\mathcal{HSD}}}$.
	To prove the theorem,  it is sufficient that $g_{\alpha,\gamma}\left[A_{i}\right]$
	is $HSD$-set. Now $\left\{ g_{\alpha,\gamma}\left[n\right]\right\} _{n\in A_{i}}$
	is an injective and stricly increasing sequence and $\left[n\alpha+\gamma\right]\leq n\alpha+1$
	implies $$\sum_{n\in A_{i}}\frac{1}{n\alpha+1}\leq\sum_{n\in A_{i}}\frac{1}{\left[n\alpha+\gamma\right]}\,.$$
	Now $\sum_{n\in A_{i}}\frac{1}{n\alpha+1}$ diverges,  when $\sum_{n\in A_{i}}\frac{1}{n}$
	diverges. 
\end{proof}
As we mentioned earlier, the $\mathcal{HSD}$ family is a very important and interesting  family for additive number theory. If we denote the  set of prime numbers of $\mathbb{N}$ by $\mathbb{P}$, then a well known fact is  $$\sum_{n\in \mathbb{P}}\frac{1}{n}=\infty\,.$$ 
The famous Green-Tao theorem \cite{GT} states that the set of prime numbers contain arithmetic progression of arbitrary length.  By this discussion we have a simple but very interesting result connected with nonhomogeneous spectra of   the set of prime numbers. 
\begin{theorem}
	Let $\alpha>0$, and let $0<\gamma<1$. If we take $\mathbb{P}$ as the set of prime numbers then $g_{\alpha,\gamma}\left[\mathbb{P}\right]\in\mathcal{AP}\cap\mathcal{HSD}$.
\end{theorem} 
Another elementary result of this section is given below.
\begin{theorem}\label{Spectra of PUAD}
	Let $\alpha>0$, and let $0<\gamma<1$.Then $g_{\alpha,\gamma}\left[\mathcal{PUAD}\right]\subseteq\mathcal{PUAD}$.
\end{theorem}

\begin{proof}
	Let $A$ be a subset of $\mathbb{N}$ and $A\in\mathcal{PUAD}$ .
	As in the Theorem \ref{Spectra of HSD},  we can find $B\subseteq A$ such that $\left\{ g_{\alpha,\gamma}\left[n\right]\right\} _{n\in B}$
	is injective and strictly increasing sequence with   $B\in\mathcal{PUAD}$. So for some
	 sequence $\left\{ n_{i}\right\} $, we get $$\frac{|B\cap[1,n_{i}]|}{n_{i}}>\beta>0\,,
	\text{for all}\, i\in\mathbb{N}.$$\\ Now two cases arise:	
	Case-I: If $\alpha\leq1$, $$|B\cap[1,n_{i}]|\leq|g_{\alpha,\gamma}\left[B\right]\cap[1,n_{i}]|\,,
	\text{for all}\, i\in\mathbb{N}.$$ From which we get,
	 $$\frac{|B\cap[1,n_{i}]|}{n_{i}}\leq\frac{|g_{\alpha,\gamma}\left[B\right]\cap[1,n_{i}]|}{n_{i}}\,,
	\text{for all}\, i\in\mathbb{N}.$$
	Case-II:  if $\alpha>1$, we have, for all $i\in\mathbb{N}$ $$\frac{|g_{\alpha,\gamma}\left[B\right]\cap[1,\left[n_{i}\alpha+\gamma\right]]|}{\left[n_{i}\alpha+\gamma\right]}=\frac{|B\cap[1,n_{i}]|}{\left[n_{i}\alpha+\gamma\right]}\geq\frac{1}{r}\frac{|B\cap[1,n_{i}]|}{n_{i}}$$
	when $r\in\mathbb{N}$ and $r>\alpha$. By combining these two cases we get our required result.
\end{proof}

The following can be proved in the same way as the above theorem is proved.
\begin{theorem}\label{spectra of PUBD}
	Let $\alpha>0$, and let $0<\gamma<1$. Then $g_{\alpha,\gamma}\left[\mathcal{PUBD}\right]\subseteq\mathcal{PUBD}$.
\end{theorem}

An important result of \cite{HJ} is the following \cite[Theorem 4.6]{HJ}:
\begin{theorem}\label{Spectrea of J-set}
		Let $\alpha>0$, and let $0<\gamma<1$.
	 If A is $J$-set in $\mathbb{N}$, then $g_{\alpha,\gamma}\left[A\right]$ is $J$-set in $\mathbb{N}$ i.e., $g_{\alpha,\gamma}\left[\mathcal{J}\right]\subseteq\mathcal{J}$.
\end{theorem}

  The proofs of Theorems \ref{spectra of PS and AP}, \ref{Spectra of HSD}, \ref{Spectra of PUAD}, and \ref{spectra of PUBD}  are elementary but proof of the Theorem \ref{Spectrea of J-set} is algebraic. To prove the Theorem \ref{Spectrea of J-set} in \cite{HJ}, Hindman and Johnson considered $J$-set as a member of ultrafilter in $J(\mathbb{N})$ and therefore used Axiom of Choice. They also suspected that the proof of the Theorem \ref{Spectrea of J-set}  ca be proved using the ramsay property of $\mathcal{J}$ family. Now, using the results that we have discussed in this section so far, we get the following   that  will be used  repeatedly  in the next section.

  \begin{theorem}\label{spectra all}
  Let $\alpha>0$, and let $0<\gamma<1$.	If $\mathcal{F}$ $=$ $\mathcal{IF}$,$\mathcal{AP}$,$\mathcal{PS}$,
  	$\mathcal{HSD}$, $\mathcal{J}$, $\mathcal{PUAD}$ or $\mathcal{PUBD}$ then $g_{\alpha,\gamma}\left[\mathcal{F}\right]\subseteq\mathcal{F}$.
  \end{theorem}

\section{Algebraic proof}\label{Algebraic results}
According to the notations in \cite[Section 3]{HJ}, We consider the circle group $\mathbb{T}$ to be $\mathbb{R}/\mathbb{Z}$
and represent the points of $\mathbb{T}$ by points in the interval
$[-\frac{1}{2},\frac{1}{2})$. Given $\alpha>0$, we let $h_{\alpha}=g_{\alpha,1/2}$,
so that for $x\in\mathbb{N}$, $h_{\alpha}$ is the nearest integer
to $\alpha x$. We define $f_{\alpha}:\mathbb{N}\rightarrow\mathbb{T}$
by, for $x\in\mathbb{N}$, $f_{\alpha}\left(x\right)=\alpha x-h_{\alpha}\left(x\right)$
and we let $\widetilde{f}_{\alpha}:\beta\mathbb{N}\rightarrow\mathbb{T}$
be the continuous extension of $f_{\alpha}$. similarly $\widetilde{g_{\alpha,\gamma}}:\beta\mathbb{N}\rightarrow\beta\mathbb{N}$
and $\widetilde{h_{\alpha}}:\beta\mathbb{N}\rightarrow\beta\mathbb{N}$
are the continuous extension of $g_{\alpha,\gamma}$ and $h_{\alpha}$
respectively.
\begin{definition}\label{Zalpha}
	Let $\alpha>0$. Then $Z_{\alpha}=\left\{ p\in\beta\mathbb{N}:\widetilde{f}_{\alpha}\left(p\right)=0\right\} $.
	
\end{definition}

	We state the following Lemma from \cite[Lemma 3.2]{HJ}, which is very important and essential result  of this section.

\begin{lemma}\label{g_alpha_ghama-p=z_alpha_p}
	Let $\alpha>0$ and let $0<\gamma<1$. The function $\widetilde{f}_{\alpha}$
	is a homomorphism from $\beta\mathbb{N}$ into $\mathbb{T}$ and consequently,
	each idempotent is in $Z_{\alpha}$. The restriction of $\widetilde{h_{\alpha}}$
	to $Z_{\alpha}$ is an isomorphism onto $Z_{1/\alpha}$ whose inverse
	is the restriction of $\widetilde{h_{1/\alpha}}$ to $Z_{1/\alpha}$.
	For all $p\in Z_{\alpha}$, $\widetilde{g_{\alpha,\gamma}}\left(p\right)=\widetilde{h_{\alpha}}\left(p\right)$.
\end{lemma}

For $\alpha>0$, and $0<\gamma<1,$ from \cite[Corollary 4.7]{HJ},
we get $\widetilde{g_{\alpha,\gamma}}\left[J\left(\mathbb{N}\right)\right]\subseteq J\left(\mathbb{N}\right)$.
As we know that $J\left(\mathbb{N}\right)=\beta\left(\mathcal{J}\right)$,
and this motivates us to prove the following theorem.
\begin{theorem}\label{g_alpha,ghama(beta_f)=beta-f}
	Let $\alpha>0$, and let $0<\gamma<1$. Let $\mathcal{F}$ be a family
	with Ramsey property of $\mathbb{N}$. Then $\widetilde{g_{\alpha,\gamma}}\left[\beta\left(\mathcal{F}\right)\right]\subseteq\beta\left(\mathcal{F}\right)$
	if $g_{\alpha,\gamma}\left[\mathcal{F}\right]\subseteq\mathcal{F}$.
\end{theorem}

\begin{proof}
	Let $p\in\beta\left(\mathcal{F}\right)$. To prove the theorem, it is sufficient to show  $\widetilde{g_{\alpha,\gamma}}\left(p\right)\subseteq \mathcal{F}$. Let $A\in\widetilde{g_{\alpha,\gamma}}\left(p\right)$, which implies $g_{\alpha,\gamma}^{-1}\left[A\right]\in p$. Taking $B=g_{\alpha,\gamma}^{-1}\left[A\right]$, we get $g_{\alpha,\gamma}\left[B\right]\subseteq A$.
	As $B$ is a $\mathcal{F}$-set, by the given  condition $g_{\alpha,\gamma}\left[B\right]$
	is a $\mathcal{F}$-set. So $A$ is a $\mathcal{F}$-set because
	of upward hereditary.
\end{proof}
For special case we get the following corollary from the above theorem.
\begin{corollary}
	If $\mathcal{F}$ $=$ $\mathcal{IF}$,$\mathcal{AP}$,$\mathcal{PS}$,
	$\mathcal{HSD}$, $\mathcal{J}$, $\mathcal{PUAD}$ or $\mathcal{PUBD}$,  then we have
	$\widetilde{g_{\alpha,\gamma}}\left[\beta\left(\mathcal{F}\right)\right]\subseteq\beta\left(\mathcal{F}\right)$.
\end{corollary}

\begin{proof}
	Follows from the fact that for $\mathcal{F}$ $=$  $\mathcal{IF}$,$\mathcal{AP}$,$\mathcal{PS}$,
	$\mathcal{HSD}$, $\mathcal{J}$, $\mathcal{PUAD}$ or $\mathcal{PUBD}$,  $\mathcal{F}$
	is a family with Ramsey property and $g_{\alpha,\gamma}\left[\mathcal{F}\right]\subseteq\mathcal{F}$ from the Theorem \ref{spectra all}.
\end{proof}
 In \cite{LL}, Li and Liang proved the following result by dynamical approach that is the main result of this  section. The technique of the proof of  the following corollary is   same as \cite[Corollary 4.8]{HJ}.
\begin{corollary}
	Let $\alpha>0$, and let $0<\gamma<1$. Let $\mathcal{F}$ be a shift
	invariant family with Ramsey property of $\mathbb{N}$
	such that $g_{\alpha,\gamma}\left[\mathcal{F}\right]\subseteq\mathcal{F}$.
	If $A$ is an essential $\mathcal{F}$-set in $\mathbb{N}$, then $g_{\alpha,\gamma}\left[A\right]$
	is an essential $\mathcal{F}$-set.
\end{corollary}

\begin{proof}
	By the definition of essential $\mathcal{F}$ set, pick  an idempotent $p\in\overline{A}\cap\beta\left(\mathcal{F}\right)$.
	Then $g_{\alpha,\gamma}\left[A\right]\in\widetilde{g_{\alpha,\gamma}}\left(p\right)$. It is sufficient to prove that $\widetilde{g_{\alpha,\gamma}}\left(p\right)$ is an idempotent. 
	By  Theorem \ref{g_alpha,ghama(beta_f)=beta-f},  $\widetilde{g_{\alpha,\gamma}}\left(p\right)\in\beta\left(\mathcal{F}\right)$.
	 Since $p$ is an idempotent and is therefore in
	$Z_{\alpha}$, while by Lemma \ref{g_alpha_ghama-p=z_alpha_p}, $\widetilde{g_{\alpha,\gamma}}\left(p\right)=\widetilde{h_{\alpha}}\left(p\right)$
	and $\widetilde{h_{\alpha}}\left(p\right)$ is an idempotent. Hence
	$g_{\alpha,\gamma}\left[A\right]$ is an essential $\mathcal{F}$-set.
	
	For special case we get the following corollary:
\end{proof}
\begin{corollary}
	Let $\alpha>0$, and let $0<\gamma<1$. For $\mathcal{F}$ $=$  $\mathcal{IF}$,$\mathcal{AP}$,$\mathcal{PS}$,
	$\mathcal{HSD}$, $\mathcal{J}$, $\mathcal{PUAD}$  or $\mathcal{PUBD}$.  If $A$ is  an essential $\mathcal{F}$-set in $\mathbb{N}$,  then $g_{\alpha,\gamma}\left[A\right]$ is an essential $\mathcal{F}$-set.
\end{corollary}

\begin{proof}
	Follows from the fact that for $\mathcal{F}$ $=$  $\mathcal{IF}$, $\mathcal{AP}$, $\mathcal{PS}$,
	$\mathcal{HSD}$, $\mathcal{J}$,  $\mathcal{PUAD}$ or $\mathcal{PUBD}$ , $\mathcal{F}$
	is a shift invariant   family with Ramsey property and $g_{\alpha,\gamma}\left[\mathcal{F}\right]\subseteq\mathcal{F}$ by the Theorem \ref{spectra all}.
\end{proof}
Above corollary implies that, $IP$-sets , $C$-sets, $D$-sets and quasi
central sets are preserved under nonhomogeneous spectra.
\begin{theorem}
	Let $\alpha>0$, and let $0<\gamma<1$. Let $\mathcal{F}$ be  a shift
	invariant family with Ramsey property of  $\mathbb{N}$ such that $g_{\alpha,\gamma}\left[\mathcal{F}\right]\subseteq\mathcal{F}$.
	If $A$ is an essential $\mathcal{F}^{\star}$-set in $\mathbb{N}$, 
	then $g_{\alpha,\gamma}\left[A\right]$ is an essential $\mathcal{F}^{\star}$-set in $\mathbb{N}$.
\end{theorem}
\begin{proof}
	 To prove the theorem, we need to show that $g_{\alpha,\gamma}\left[A\right]$ is a member
	of every idempotent in $\beta\left(\mathcal{F}\right)$.
	 Let  $p$ be  an idempotent in  $ \beta\left(\mathcal{F}\right)$. 
	 By Lemma \ref{g_alpha_ghama-p=z_alpha_p} and  Theorem \ref{g_alpha,ghama(beta_f)=beta-f}
    	we get $\widetilde{h_{1/\alpha}}\left(p\right)$ is an idempotent
	in $ \beta\left(\mathcal{F}\right)$, so $A\in\widetilde{h_{1/\alpha}}\left(p\right)$.
	Again by  Lemma \ref{g_alpha_ghama-p=z_alpha_p}  $\widetilde{g_{\alpha,\gamma}}\left(\widetilde{h_{1/\alpha}}\left(p\right)\right)=p$
	and so $g_{\alpha,\gamma}\left[A\right]\in p$.
\end{proof}
The following corollary follows from the above theorem and Theorem \ref{spectra all}.
\begin{corollary}
	Let $\alpha>0$, and let $0<\gamma<1$. For $\mathcal{F}$  $=$ $\mathcal{IF}$,$\mathcal{AP}$,$\mathcal{PS}$,
	$\mathcal{HSD}$, $\mathcal{J}$,  $\mathcal{PUAD}$ or $\mathcal{PUBD}$. If
	$A$ is an essential $\mathcal{F}^{\star}$-set in $\mathbb{N}$, Then
	$g_{\alpha,\gamma}\left[A\right]$ is an essential $\mathcal{F}^{\star}$-set.
\end{corollary}

From the above corollary, it is clear that $IP^{\star}$-sets, $C^{\star}$-sets,
$D^{\star}$-sets and  quasi central$^{\star}$ sets are preserved under
nonhomogeneous spectra.

This is the proper time where, we can discuss about the preservation of ideals (associated with essential $\mathcal{F}$-sets) of $\beta\mathbb{N}$ by $\widetilde{g_{\alpha,\gamma}}$. To do so, 
We state the following theorem from \cite[Theorem 5.2]{HJ},  which will be used in the next theorem.
\begin{theorem}\label{Preservation of ideal}
	Let $I$ be a subset of $\beta\mathbb{N}$ which is a left ideal of
	$\left(\beta\mathbb{Z},+\right)$ and assume that whenever $0<\alpha$
	and $0<\gamma<1$, one has $\widetilde{g_{\alpha,\gamma}}\left[I\right]\subseteq I$.
	Let $0<\alpha\leq1$ and let $0<\gamma<1$. Then $\widetilde{g_{\alpha,\gamma}}\left[I\right]=I$.
\end{theorem}

In \cite[Theorem 5.2, Theorem 5.4]{HJ}, Hindman and Johnson proved
that if $\alpha\leq1$ and $0<\gamma<1$, then the ideals $K\left(\beta\mathbb{N}\right)$,
$clK\left(\beta\mathbb{N}\right)$ and $J\left(\mathbb{N}\right)$
are preserved by  $\widetilde{g_{\alpha,\gamma}}$ but fail surprisingly. 
to be preserved if $\alpha>1$.
\begin{theorem}\label{ideal preservation and equal}
	Let $0<\alpha\leq1$, and let $0<\gamma<1$. Let $\mathcal{F}$ be
	a shift and inverse shift invariant family with Ramsey property of
	$\mathbb{N}$ such that $g_{\alpha,\gamma}\left[\mathcal{F}\right]\subseteq\mathcal{F}$.
	Then $\widetilde{g_{\alpha,\gamma}}\left[\beta\left(\mathcal{F}\right)\right]=\beta\left(\mathcal{F}\right)$.
\end{theorem}

\begin{proof}
	     From  Theorem \ref{g_alpha,ghama(beta_f)=beta-f} we get $\widetilde{g_{\alpha,\gamma}}\left[\beta\left(\mathcal{F}\right)\right]\subseteq\beta\left(\mathcal{F}\right)$  then  according to the Theorem \ref{Preservation of ideal}, it is sufficient to show that $\beta\left(\mathcal{F}\right)$
	is a left ideal of $\beta\mathbb{Z}$. Let $p\in\beta\left(\mathcal{F}\right)$.
	Then $\beta\mathbb{Z}+p=cl\left(\mathbb{Z}+p\right)$. So it is sufficient
	to show that $\mathbb{Z}+p\subseteq\beta\left(\mathcal{F}\right)$.
	Let $m\in\mathbb{Z}$ and let $A\in m+p$. Which implies $-m+A\in p$
	and $-m+A\in\mathcal{F}$. As $\mathcal{F}$ is a shift and inverse
	shift invariant, we have $A\in\mathcal{F}$.
\end{proof}
Let $\alpha>1$ and let $0<\gamma<1$. If $I$ is a left ideal of $\beta\mathbb{N}$, then from \cite[Theorem 4.5]{HJ}, We get $I\setminus g_{\alpha,\gamma}\left[\beta\mathbb{N}\right]\neq\emptyset$. So, the conclusion of the above theorem is not true if $\alpha>1$.
\begin{corollary}
	Let $0<\alpha\leq1$, and let $0<\gamma<1$. If $\mathcal{F}$ is
	any of   $\mathcal{IF}$, $\mathcal{AP}$, $\mathcal{PS}$, $\mathcal{HSD}$, 
	$\mathcal{J}$, $\mathcal{PUAD}$ or $\mathcal{PUBD}$, then $\widetilde{g_{\alpha,\gamma}}\left[\beta\left(\mathcal{F}\right)\right]=\beta\left(\mathcal{F}\right)$.
\end{corollary}

\begin{proof}
	Follows from the fact that  $\mathcal{IF}$, $\mathcal{AP}$, $\mathcal{PS}$, 
	$\mathcal{HSD}$, $\mathcal{J}$ , $\mathcal{PUAD}$ and  $\mathcal{PUBD}$ are all shift
	and inverse shift invariant family with Ramsey property of $\mathbb{N}$
	and $g_{\alpha,\gamma}\left[\mathcal{F}\right]\subseteq\mathcal{F}$ by \ref{spectra all}.
\end{proof}
\begin{corollary}
	Let $0<\alpha\leq1$, and let $0<\gamma<1$. Let $\mathcal{F}$ is
	a shift and inverse shift invariant family with Ramsey property of
	$\mathbb{N}$ such that $g_{\alpha,\gamma}\left[\mathcal{F}\right]\subseteq\mathcal{F}$.
	Then $g_{\alpha,\gamma}\left[\mathcal{F}^{\star}\right]\subseteq\mathcal{F}^{\star}$.
\end{corollary}

\begin{proof}
	Let $A\in\mathcal{F}^{\star}$-set. Then $\beta\left(\mathcal{F}\right)\subseteq\overline{A}$
	so by Theorem \ref{ideal preservation and equal}, $\beta\left(\mathcal{F}\right)=\widetilde{g_{\alpha,\gamma}}\left[\beta\left(\mathcal{F}\right)\right]\subseteq\widetilde{g_{\alpha,\gamma}}\left[\overline{A}\right]=\overline{g_{\alpha,\gamma}\left[A\right]}$.
\end{proof}
\begin{corollary}
	Let $0<\alpha\leq1$, and let $0<\gamma<1$. If  $\mathcal{F}$ is
	any of  $\mathcal{IF}$, $\mathcal{AP}$, $\mathcal{PS}$, $\mathcal{HSD}$,
	$\mathcal{J}$, $\mathcal{PUAD}$  or  $\mathcal{PUBD}$, then $g_{\alpha,\gamma}\left[\mathcal{F}^{\star}\right]\subseteq\mathcal{F}^{\star}$.
\end{corollary}

\begin{proof}
	Follows from the fact that  $\mathcal{IF}$, $\mathcal{AP}$, $\mathcal{PS}$,
	$\mathcal{HSD}$, $\mathcal{J}$, $\mathcal{PUAD}$ and  $\mathcal{PUBD}$ are all shift
	and inverse shift invariant family with Ramsey property of $\mathbb{N}$ and
	$g_{\alpha,\gamma}\left[\mathcal{F}\right]\subseteq\mathcal{F}$.
\end{proof}
\section{Question on elementary proof}
In \cite{L}, Li, established dynamical  characterization of essential $\mathcal{F}$-sets for $S=\mathbb{N}$ and also found an elementary characterization using dynamical  characterization.
In \cite[Theorem 5.2.3]{C}, Christopherson  established dynamical characterization
of essential $\mathcal{F}$-sets for arbitrary semigroup. Although elementary characterization
of quasi central-sets and $C$-sets are known from \cite[Theorem 3.7]{HMS}
and \cite[Theorem 2.7]{HS2} respectively.  Quasi central
sets and $C$ sets are coming from by the setting of essential $\mathcal{F}$-set
and this fact confines that,  essential $\mathcal{F}$-sets
might have elementary characterization for arbitrary semigroup.  Elementary characterization
of essential $\mathcal{F}$-sets from algebraic characterization are known from \cite[Theorem 5]{DDG} for arbitrary semigroup.

\begin{definition} Let $\omega$ be the first infinite ordinal and each ordinal
	indicates the set of all it's predecessor. In particular, $0=\emptyset,$
	for each $n\in\mathbb{N},\:n=\left\{ 0,1,...,n-1\right\} $.
	\begin{itemize}	
        \item[(a)] If $f$ is a function and $dom\left(f\right)=n\in\omega$,
	then for all $x$, $f^{\frown}x=f\cup\left\{ \left(n,x\right)\right\} $.
		\item[(b)] Let $T$ be a set functions whose domains are members
	of $\omega$. For each $f\in T$, $B_{f}\left(T\right)=\left\{ x:f^{\frown}x\in T\right\}$.
\end{itemize}
\end{definition}	
	We get the following theorem from \cite[Theorem 5]{DDG}.

\begin{theorem}\label{Elementary F sets}
	Let $\left(S,.\right)$ be a semigroup, and assume that $\mathcal{F}$
	is a family of subsets of $S$ with the Ramsay property such that
	$\beta\left(\mathcal{F}\right)$ is a subsemigroup of $\beta S$.
	Let $A\subseteq S$. Statements (a), (b) and (c) are equivalent and
	are implied by statement (d). If $S$ is countable, then all the five
	statements are equivalent.
	
	\begin{itemize}
			
\item[(a)] $A$ is an essential $\mathcal{F}$-set.
	
\item[(b)] There is a non empty set $T$ of function such that
\begin{itemize}

	\item[(i)]  For all $f\in T$,$\text{domain}\left(f\right)\in\omega$
	and $rang\left(f\right)\subseteq A$;
	
	\item[(ii)]  For all $f\in T$ and all $x\in B_{f}\left(T\right)$,
	$B_{f^{\frown}x}\subseteq x^{-1}B_{f}\left(T\right)$; and
	
	\item[(iii)]  For all $F\in\mathcal{P}_{f}\left(T\right)$, $\cap_{f\in F}B_{f}(T)$
	is a $\mathcal{F}$-set.
	
	\end{itemize}

\item[(c)]  There is a downward directed family $\left\langle C_{F}\right\rangle _{F\in I}$
of subsets of $A$ such that
\begin{itemize}
	
\item[(i)] for each $F\in I$ and each $x\in C_{F}$ there exists
$G\in I$ with $C_{G}\subseteq x^{-1}C_{F}$ and

\item[(ii)] for each $\mathcal{F}\in\mathcal{P}_{f}\left(I\right),\,\bigcap_{F\in\mathcal{F}}C_{F}$
is a $\mathcal{F}$-set. 

\end{itemize}

\item[(d)] There is a decreasing sequence $\left\langle C_{n}\right\rangle _{n=1}^{\infty}$
of subsets of $A$ such that 

\begin{itemize}

\item[(i)] for each $n\in\mathbb{N}$ and each $x\in C_{n}$, there
exists $m\in\mathbb{N}$ with $C_{m}\subseteq x^{-1}C_{n}$ and 

\item[(ii)] for each $n\in\mathbb{N}$, $C_{n}$ is a $\mathcal{F}$-set.
\end{itemize}

\end{itemize}
\end{theorem}
The following question comes naturally.
\begin{question}
	Is it possible  to obtain elementary characterization of essential $\mathcal{F}$-sets from dynamical characterization for arbitrary semigroup?
\end{question}

 Let $\mathcal{F}$  be a shift invarient family with Ramsey property. In \cite{LL}, Li and Liang, proved the preservation of nonhomogeneous spectra of  essential $\mathcal{F}$-sets by dynamical characterization of essential $\mathcal{F}$-sets in $\mathbb{N}$ and we have proved the same result in this article using algebraic characterization of essential $\mathcal{F}$-sets.  From this discussion, the following question arises naturally as we know the  elementary characterization of essential $\mathcal{F}$-sets.

\begin{question}
	Let $\mathcal{F}$  be a shift invarient family with Ramsey property. Is it possible  to prove the preservation of nonhomogeneous spectra of  essential $\mathcal{F}$-sets by elementary characterization of essential $\mathcal{F}$-sets in $\mathbb{N}$?
\end{question}
Although we do not  know the answer of the above question, we conclude this section with a discussion that is quite directly connected to this question. 

Let A be an essential $\mathcal{F}$-sets of $\mathbb{N}$ then $$A\supseteq A_{1}\supseteq A_{2}\supseteq\ldots\supseteq A_{n}\supseteq\ldots$$ such that for each $n\in\mathbb{N}$  and each  $x\in A_{n}$, there exists $m\in\mathbb{N}$ with  $ x+A_{m}\subseteq A_{n}$ and $A_{i}$ are essential  $\mathcal{F}$ sets for all  $i\in\mathbb{N}$. $$g_{\alpha,\gamma}\left[A\right]\supseteq g_{\alpha,\gamma}\left[A_{1}\right]\supseteq g_{\alpha,\gamma}\left[A_{2}\right]\supseteq\ldots\supseteq g_{\alpha,\gamma}\left[A_{n}\right]\supseteq\ldots$$ 
and each $g_{\alpha,\gamma}\left[A_{i}\right]$ are essential $\mathcal{F}$-sets for all $i\in\mathbb{N}$.
\begin{question}
	Does the above chain $\left\langle g_{\alpha,\gamma}\left[A_{i}\right]\right\rangle _{i=1}^{\infty}$  satisfies  Theorem \ref{Elementary F sets} (d)?
\end{question}

\bibliographystyle{plain}

\begin{thebibliography}{9}
	\bibitem{B} T. Bang, On the sequence $\left[n\alpha\right]$,
	$n=1,2,\ldots,$Math. Scand. 5(1957), 69-76.
	
	\bibitem{BD} V. Bergelson and T. Downarowicz, Large sets of
	integers and hierarchy of mixing properties of measure preserving
	systems, Colloq. Math. 110 (2008), 117-150.
	
	\bibitem{BH} V. Bergelson and N. Hindman, Nonmetrizable topological
	dynamics and Ramsay Theory, Trans. Amer. Math. Soc. 320 (1990), 293-320.
	
	\bibitem{BHK}V. Bergelson, N. Hindman and B. Kra, Iterated spectra
	of numbers-elementary, dynamically and algebraic approaches, Trans.
	Amer. Math. Soc. 348 (1996),893-912.
	
	\bibitem{C} C.Christopherson, Closed ideals in the Stone-\v{C}ech
	compactification of a countable semigroup and some application to
	ergodic theory and topological dynamics, PhD thesis, Ohio State University,
	2014.
	
    \bibitem{DDG}	D. De, P. Debnath, and G. Goswami, Elementary characterization of essential $\mathcal{F}$-sets and its combinatorial consequences, Semigroup Fourum, 104 (2022), 45-57.
	
	
	\bibitem{DHS} D. De, N. Hindman, and D. Strauss, A new and stronger
	Central Sets Theorem, Fundamenta Mathematicae 199 (2008), 155-175.
	
	\bibitem{F} H. Furstenberg, Recurrence in ergodic theory and
	combinatorical number theory, Princeton University Press, Princeton,
	1981.
	
	\bibitem{GLL} R. Graham, S. Lin and C. Lin, Spectra of numbers,
	Math. Mag. 51(1978), 174-176.
	
	\bibitem{GT} B. Green and T.Tao, The primes contain arithmetic progression of arbitrary length, Annals of Mathematics, 167 (2008), 481-547.
	
	\bibitem{HaW} G. Hardy and E. Wright, An introduction to the
	theory of numbers, Fifth edition, Oxford University Press, London,
	1979.
	
	\bibitem{Hi} N. Hindman, Some equivalent of the Erdős sum of reciprocal conjecture, Europ. J. Combinatorics 9 (1988)  39-47.
	
	\bibitem{HJ} N. Hindman and J.H. Johnson, Image of $C$-sets
	and related large sets under nonhomogeneous spectra, \#A2 Integers
	12B (2012/13).
	
	\bibitem{HMS} N. Hindman, A. Maleki, and D. Strauss, Central
	sets and their combinatorial characterization, J. Comb. Theory (Series
	A) 74 (1996), 188-208.
	
	\bibitem{HS} N. Hindman and D. Strauss, Algebra in the Stone-\v{C}ech
	compactication: theory and applications, second edition, de Gruyter,
	Berlin, 2012.
	
	\bibitem{HS2} N. Hindman and D. Strauss, A simple characterization of sets satisfying the central sets theorem, New York J. math. 15 (2009), 405-413.
	
	\bibitem{L} J. Li,  Dynamical characterization of $C$-sets and its application, Fundamental Mathematicae 216 (2012), 259-286.
	
	\bibitem{LL} J. Li and X. Liang, A dynamical approach to nonhomogeneous spectra, Fundamental Mathematicae 262 (2023), 221-233.
	
	\bibitem{ShY} H. Shi and H. Yang, Nonmetrizable topological
	dynamical characterization of central sets, Fund. Math. 150(1996),
	1-9.
	
	\bibitem{Sk} T. Skolem, On certain distributions of integers
	in pairs with given differences, Math. Scand. 5 (1957), 57-68.
	
	\bibitem{Sk2} T. Skolem, $\ddot{U}$ber einige Eigenschaften
	der Zahlenmengen $\left[\alpha n+\beta\right]$ bei irrationalem $\alpha$
	mit einleitenden Bemerkungen $\ddot{u}$ber einige kombinatorische
	problem, Norske Vid. Selsk. Forh. 30(1957), 42-49.
	
	
	
\end{thebibliography}

\end{document}